\documentclass[11pt]{amsart}
\usepackage{amsmath}
\usepackage{amsfonts}
\usepackage{amssymb}
\usepackage{amsthm}
\usepackage[english]{babel}
\usepackage{hyperref}
\hypersetup{
	colorlinks=true,
	linkcolor=black,
	citecolor=black,
	urlcolor=[rgb]{0.05,0,0.25},
	}
\makeatletter
\@namedef{subjclassname@2020}{%
  \textup{2020} Mathematics Subject Classification}
\makeatother
\def \N{\mathbb N}
\def \Z{\mathbb Z}
\def \R{\mathbb R}

\def \T{\mathbb T}
\newcommand{\B}{{\mathcal B}}
\newcommand{\CC}{{\mathcal C}}
\newcommand{\LL}{{\mathcal L}}
\renewcommand{\Re}{\mathrm{Re\,}}
\renewcommand{\Im}{\mathrm{Im\,}}  
\let\eps=\varepsilon
\let\vphi=\varphi
\let\si=\sigma
\let\til=\widetilde
\let\hat=\widehat
\let\inter=\cap
\let\Inter=\bigcap
\DeclareMathOperator{\Gp}{\mathrm{Gp}}
\DeclareMathOperator{\supp}{\mathrm{supp}}
\def\scal(#1,#2){(#1\mid#2)}
\def\du(#1,#2){\langle#1,\,#2\rangle}

\newtheorem{theorem}{Theorem}
\newtheorem*{theorem*}{Theorem}

\newtheorem*{proposition*}{Proposition}
\newtheorem{corollary}{Corollary}
\newtheorem*{corollary*}{Corollary}
\newtheorem{lemma}{Lemma}
\newtheorem*{definition*}{Definition}
\newtheorem*{remark}{Remark}

\newtheorem*{ack*}{Acknowledgments}
    
\title{On the Foia\c s and Stratila Theorem}

\author{Fran\c cois Parreau.}
\address[Fran\c{c}ois Parreau]{
Universit\'e Sorbonne Paris Nord, LAGA, CNRS, UMR 7539, F-93430 Villetaneuse, France}
\email{parreau@math.univ-paris13.fr}

\date{20/12/2023}



\begin{document}

\begin{abstract}We extend the Foia\c s and Stratila theorem to the case of $L^2$--functions whose spectral measure is continuous and concentrated on an independent Helson set, and to ergodic actions of locally compact second countable abelian groups. We first prove it for functions satisfying Carleman's condition for 
the Hamburger moment problem, without the assumption that the spectral measure is supported by a Helson set. Then we show independently that the spectral projector associated with a Helson set preserves each $L^p$ space, with an appropriate bound of the corresponding norm.
\end{abstract}

\maketitle

\section{Introduction}
\subsection{Main results}
The Foia\c s and Stratila theorem (\cite{FS}) asserts that, given an 
ergodic measure--pre\-serv\-ing automorphism $T$ on a standard probability space, if
the spectral measure $\si$  of a non zero square-integrable complex function $f$ is continuous and
supported by a \emph{Kronecker set} (\cite{Rud_F}), then the process $(f\circ T^n)$ is
Gaussian. The dynamical system generated by the process
$(f\circ T^n)$ is then determined up to a metric isomorphism by $\sigma$,
the spectral measure of the process. Except for the case of discrete
spectrum and not deep extensions (see \cite{LP}), this is the only result
of spectral determination in ergodic theory.

It implies strong ergodic properties for these
\emph{Gaussian--Kronecker} automorphisms. Factors and self-joinings
of such systems can be completely described (\cite{Thou1}). In \cite{LPT}, 
we extend these properties and prove disjointness results for a wider class
of Gaussian automorphisms, the ``$GAG$" automorphisms, which include all
Gaussian automorphisms with simple spectrum and thus mixing cases, but all
the results there eventually rely on the Foia\c s and Stratila theorem. 

\vspace{1ex}
Our main result is its extension to
the larger class of {\em algebraically independent Helson sets}. We also extend it to ergodic actions of locally compact second countable abelian groups.

\vspace{1ex}
Let  $(X,\B,\mu)$ denote a standard probability space and let $T=(T_g)_{g\in G}$ be an action of a locally
compact second countable abelian group $G$ on $X$ by measure--preserving automorphisms (for sake of simplicity, we use the same notation as for a single measure--preserving automorphism).

The Fourier transform of a complex Borel measure $\sigma$
on the dual group $\Gamma=\hat G$ is defined by
$\hat\si(g)=\int_{\Gamma} \gamma(g)\,d\si(\gamma)$ for $g\in G$,
and the spectral measure $\si_f$ of $f\in L^2(\mu)$ is the finite positive 
Borel measure on $\Gamma$ given by 
$$\hat\si_f(g)=\scal(f\circ T_g,f)\qquad(g\in G).$$

A closed set $K\subset\Gamma$ is a \emph{Helson} set with 
constant $\alpha$ ($0<\alpha\le 1$), or a Helson--$\alpha$ set 
if, for every complex Borel measure $\si$ on $K$ ($\Vert.\Vert$
denotes the total variation norm),
$$\sup_{g\in G}\vert \hat\si(g)\vert\geq\alpha\,\Vert\si\Vert.$$
An equivalent definition, maybe more usual, is that $K$ is a Helson set if each continuous function on $K$, vanishing at infinity if $\Gamma$ is not compact, is the restriction to $K$ of the Fourier transform of an integrable function. Closed Kronecker sets are Helson--$1$ sets and are also algebraically independent (\cite{Rud_F, LiPo}).

\begin{theorem}\label{HI} Assume that $(T_g)_{g\in G}$ is ergodic and let $f$ be a 
non-zero function in $L^2(\mu)$. 
If the spectral measure of $f$ is continuous and concentrated on an 
algebraically independent Helson set, then $f$ has a Gaussian distribution.
\end{theorem}
Whereas the Kronecker assumption implies that the process  under consideration is \emph{rigid} i.e.\ such that there exists a sub-sequence $(T^{n_j})$ converging in measure to the identity, Theorem~\ref{HI} allows us to get \emph{mildly mixing} examples, but not yet strongly mixing examples.

\vspace{1ex}
We shall firstly prove a result where the assumption that $\si_f$ is supported by a Helson set 
is replaced by an additional hypothesis on the moments of $f$. 

\begin{definition*}\label{Carl} Given a positive measure $\mu$, we denote 
$\CC(\mu)$ the class of all $f\in\Inter_{2\leq p<+\infty}L^p(\mu)$ such
that $$\sum_{p=2}^{\infty}1/\Vert f\Vert_p=+\infty.$$
\end{definition*}

For real random variables, this condition is known as Carleman's condition for 
the Hamburger moment problem. 

\begin{theorem}\label{CI} Assume that $(T_g)_{g\in G}$ is ergodic and let $f$ be a 
non-zero function in $\CC(\mu)$. 
If the spectral measure of $f$ is continuous and concentrated on an
algebraically independent compact set, then $f$ has a Gaussian distribution.
\end{theorem}
\begin{corollary}\label{process} Under the assumptions of Theorem~\ref{HI} or of Theorem \ref{CI}, the functions $f\circ T_g$ (${g\in G}$) span a Gaussian space. In case of a single measure--preserving automorphism $T$, the process $(f\circ T^n)_{n\in \Z}$ is Gaussian.
\end{corollary}

Theorem \ref{HI} follows from Theorem \ref{CI} and from a result that may be of independent interest:

\begin{theorem}\label{Helsonproj} Let $K$ be a compact Helson subset of 
$\Gamma$. Then, for every $p\geq 2$, the spectral projector $\pi_K$ corresponding to $K$
maps $L^p(\mu)$ into itself with, for every $f\in L^p(\mu)$,
$$\Vert\pi_K f\Vert_p\leq C p\,\Vert f\Vert_p$$
where $C$ depends only on the Helson constant of $K$.

In particular $\pi_K(L^{\infty}(\mu))\subset \CC(\mu)$.
\end{theorem}
The last theorems highlight the different roles of algebraic independence and harmonic properties of the support of the spectral measure.
Under the only assumption that the spectral measure is continuous and supported by an 
independent compact set, which would yield mixing examples, the problem remains open and there is little hope to get a positive answer.
However, under the weaker assumption that the Gaussian automorphism corresponding to the spectral measure has a simple spectrum, there are examples where the spectral determination property does not hold.

In Section~\ref{sectZsido}, we show how to extend the proofs to the case of group actions. The proof of Theorem~\ref{CI} is given in Section~\ref{group}. Those of Theorems~\ref{Helsonproj} and \ref{HI} follow in Section~\ref{sectHelson}. Section~\ref{compl} contains the examples mentioned above and a few additions.
\subsection{Notation, definitions and preliminaries}
We refer to \cite{CoFoSi}, \cite{Parry}, \cite{Queff} for basic definitions and
results in ergodic theory and spectral theory of dynamical systems.
For harmonic analysis, we refer to \cite{Rud_F}, and to \cite{LiPo} for
definitions and properties of thin sets.

The abelian group $G$ acting 
on $(X,\B,\mu)$ is supposed locally compact second countable, and non 
compact. For $g\in G$ we also note $T_g$ the unitary operator $f\to f\circ T_g$ on $L^2(\mu)$. We shall note its dual group $\Gamma$ multiplicatively and, for $\gamma\in\Gamma$ and $g\in G$, $\gamma(g)$ is also denoted $\du(\gamma,g)$. In the case $G=\Z$ we identify $\Gamma=\T$ with $\mathbb{S}^1$.

The group generated by $K\subset \Gamma$ is denoted $\Gp(K)$. We shall need
algebraic independence through the property: if $K_1$, $K_2$ are two 
disjoint subsets of $K$, then $\Gp(K_1)$ and $\Gp(K_2)\setminus\{1\}$ are still 
disjoint; thus the definition of independence can be taken here 
in the weaker sense: when $\gamma_1,\ldots,\gamma_k\in K$ and
$n_1,\ldots,n_k\in\Z$, then
$\gamma_1^{n_1}\times\cdots\times\gamma_k^{n_k}=1$ only if 
$\gamma_1^{n_1}=\cdots=\gamma_k^{n_k}=1$ ($n_1=\cdots=n_k=0$ is not required, 
and this allows finite order elements).

Let $\B(f)$ (resp. $\B(H)$) denote the sub$-\si-$algebra generated by a
measurable function $f$ (resp. a subset $H$ of $L^2(\mu)$).  If 
$f\in L^2(\mu)$, the closed invariant subspace of $L^2(\mu)$ generated by
$f$ is denoted $Z(f)$, so the factor generated by $f$ is $\B(Z(f))$
(factors will be taken as invariant sub$-\si-$algebras).

By a measure $\si$ on $\Gamma$ we always mean a complex Borel measure on
$\Gamma$. The support of $\si$, $\supp(\si)$ will always mean its \emph{topological support}, i.e.\ the smallest closed set of $\Gamma$ on which $\vert\si\vert$ is concentrated.
We denote by $\til\si$ the measure defined by
$\til\si(B)=\bar\si(B^{-1})$ for every Borel subset $B$ of $\Gamma$; given $f\in L^2(\mu)$,
the spectral measure of $\bar{f}$ is $\til\si_f$ and it is concentrated on $\supp(\si_f)^{-1}$.
If $\si$ is the spectral type of $T$, defined up to equivalence of measures, the spectral representation yields an isometry
$\vphi\to \vphi(T)$ from $L^\infty(\si)$ onto a subalgebra of $L^2(\mu)$ such that $\vphi(T)$ corresponds to the multiplication by $\vphi$, and in particular each $g\in G$, taken as a character of $\Gamma$, corresponds to $T_g$.

Given a Borel subset $K$ of
$\Gamma$, $\mathbf{1}_K(T)$ is the spectral projector of $L^2(\mu)$
corresponding to $K$, which we denote by $\pi_{K}$. For $f\in L^2(\mu)$, we have $\overline{\pi_K f}=\pi_{K^{-1}}\bar f$.

\vspace{1ex}
In our proofs, the operation of $A(\Gamma)$ on the spaces $L^p(\mu)$ plays a major role.
Recall that $A(\Gamma)$ denotes the Banach algebra of Fourier transforms of integrable functions, equipped with the norm inherited from the $L^1(G)$ norm. Given $\vphi\in A(\Gamma)$, the operator $\vphi(T)$ sends each $L^p(\mu)$  ($1\le p\le +\infty$) into itself, with norm $\Vert\vphi(T)\Vert_{\LL(L^p(\mu))}\le \Vert\vphi\Vert_{A(\Gamma)}$, and thus it also preserves $\CC(\mu)$.

\vspace{2ex}
Concerning Gaussian automorphisms, we shall use the 
definitions and notation of \cite{LPT}. In particular, if $\si$ is a 
continuous symmetric measure on $\T$, we denote 
$T_{\si}$ the Gaussian automorphism defined by the real Gaussian 
process of spectral measure $\si$. However, as Foia\c s and 
Stratila in \cite{FS}, we consider complex--valued Gaussian processes
$(f\circ T^n)$, which will be more convenient.

Since a $L^2-$limit of Gaussian functions still is Gaussian, we have 
the elementary lemma, which shows that it is sufficient to prove Theorem \ref{HI} in the  case 
when $\si_f$ has compact support:
\begin{lemma}\label{FSlimits}
Let $f\in L^2(\mu)$. If $(K_n)$ is a sequence of Borel subsets of $\Gamma$ such that 
$\si_{f}(\Gamma\setminus K_n)\to 0$ and each $\pi_{K_n}f$ is Gaussian, then $f$ is Gaussian. 
\end{lemma}

Furthermore, for the proofs 
of Theorems \ref{HI} and \ref{CI}, we can restrict 
ourselves to countable group actions. Indeed, let $G_0$ 
be a countable dense subgroup of $G$, endowed with the discrete 
topology. Then $\Gamma$ is countinuously embedded
in the compact group $\Gamma_0=\hat G_0$, the spectral measures of $f$ 
for the actions of $G$ and $G_0$ being identified in this embedding.
Clearly, if a set $K$ in $\Gamma$ is independent and compact
in $\Gamma$, these properties still hold in $\Gamma_0$.
The Fourier transform of a measure concentrated on $K$ is 
continuous in the topology of $G$, hence by density of $G_0$ in $G$,
directly from the definition, if $K$ is Helson in $\Gamma$, the same 
holds in $\Gamma_0$.

\vspace{1ex}
So, {\em we henceforth assume that the abelian group $G$ is countable and 
discrete, so that $\Gamma$ is compact and metrizable.}

\vspace{1ex}
Note also that, given  $\eps>0$,
we can choose a totally disconnected compact subset $K$ with 
$\si_f(\Gamma\setminus K)<\eps$; indeed this is a standard fact for 
finite positive measures on $\T^{\N}$ and the dual group of a countable abelian group
is naturally embedded as a compact subgroup of $\T^n$. By Lemma 
\ref{FSlimits}, if we obtain that every such $\pi_Kf$ is 
Gaussian, then $f$ itself is Gaussian.

\section{Zsido's theorem and abelian group actions}\label{sectZsido}
\subsection{Spectral process}\ \\
 Let $f$ be a non-zero function in $L^2(\mu)$.
We consider firstly the case of a single measure-preserving
automorphism $T$. Let $\gamma(t)=\exp(2\pi it)$ for
$t\in [0,1]$. The spectral process corresponding to $f$ is defined by
$$f_t=\pi_{\gamma([0,t))}f\qquad(t\in [0,1]).$$
As for the proof of C.~Foia\c{s} and S.~Stratila, the main argument in the proofs 
of Theorems~\ref{HI} and \ref{CI} will be to prove 
that this process has independent increments. In order to obtain that 
it is a Gaussian process, it is then sufficient to know that it admits 
a version with a.s.\ continuous sample functions (see e.g. \cite{Doob}, 
chap.\ VII, Theorem 7.1). This latter fact has been proved independently by
L.~Zsido (\cite{Zs}):

\begin{theorem*}[Zsido] Assume that $T$ is ergodic and that 
$\si_f$ is continuous. If the spectral process $(f_t)$ has
independent increments, then it is Gaussian.
\end{theorem*}
For a countable abelian group action $(T_g)_{g\in G}$
we want to construct a process with similar properties.

The following lemma will apply when $K=\supp\si_f$ is totally disconnected. 
Then $K$ is homeomorphic to a compact set $L$ of $[0,1]$. We can 
furthermore assume that $L$ contains the endpoints $0$ and $1$. Choose a
homeomorphism $\gamma$ from $L$ onto $K$ and let
\begin{equation}\label{specproc}
f_t=\pi_{\gamma\bigl([0,t)\inter L\bigr)}f\qquad(t\in [0,1]).
\end{equation}
\begin{lemma}\label{speccont}
Assume that the action $(T_g)$ is ergodic, that $\sigma_f$ is
continuous with compact support $K$ homeomorphic to $L\in[0,1]$,
and let $(f_t)$ be defined by \emph{(\ref{specproc})}. Then, for every
$t\in(0,1)$, the spectral measures $\si_{f_t}$ and $\si_{f_1-f_t}$ are
concentrated on disjoint open sets of $K$. Moreover, if this 
process has independent increments, then it
admits a version with almost surely continuous sample functions 
and thus it is Gaussian. 
\end{lemma}
\begin{proof} The first assertion is immediate from the facts that 
$$f_1-f_t=\pi_{\gamma\bigl([t,1]\inter L\bigr)}f=\pi_{\gamma\bigl((t,1]\inter L\bigr)}f$$
since $\si_f$ is continuous, and that $\gamma\bigl([0,t)\inter L\bigr)$ and
$\gamma\bigl((t,1]\inter L\bigr)$ are disjoint open sets of $K$.

For the second assertion, we proceed by slight modifications of Zsido's
proof \cite{Zs}. By the same classical argument of probability theory
(see \cite{Doob}), $(f_t)$ has a version whose sample functions
$t\mapsto(f_t(x))$ are a.e.\ 
right continuous with left-limits (\emph{c\`{a}dl\`{a}g}) and thus have a
 jump $\Delta f_t(x)$ at every $t\in[0,1]$. Since the 
complement of $L$ consists of countably many open intervals $(t'_n,t''_n)$,
we can moreover let $f_t(x)=f_{t'_{n}}(x)$ everywhere on each of these
intervals, so that the sample functions are continuous on 
$[0,1]\setminus L$.

The main step is to show that, given $g\in G$, there is a set $F$ of
full measure in $X$ such that for every $x\in F$
\begin{equation}\label{sauts}
\Delta f_t(T_gx)=\du(\gamma(t),g)\,\Delta f_t(x)
\qquad\text{for every }t\in L
\end{equation}
Let $\eta>0$. Since the map $t\mapsto \du(\gamma(t),g)$ 
is continuous on $L$, we can find a finite subdivision $t_0=0<t_1<\cdots<t_k=1$ of
points in $L$ such that $\vert\du(\gamma(t),g)-\du(\gamma(t_j),g)\vert<\eta$
on each $[t_{j-1},t_j]\inter L$. Let then 
$f_j=f_{t_j}-f_{t_{j-1}}=\pi_{\gamma\left([t_{j-1},t_j)\inter 
L\right)}f$,
\mbox{($1\leq j\leq k$)}. Since $T_g$ corresponds in the spectral representation
to the multiplication by $\du(\cdot,g)$, it follows
$$\Vert f_j\circ T_g-\du(\gamma(t_j),g)f_j\Vert_2\leq\eta\,\Vert f_j\Vert_2,$$
hence the set $N'_j$ of all $x\in X$ where
$$\vert f_j(T_gx)-\du(\gamma(t_j),g)f_j(x)\vert\geq\eta^{1/2}$$
has measure${}\leq\eta\,\Vert f_j\Vert^2$.

Moreover, the set $N''_j$ of all $x$ where $\vert f_j(x)\vert\geq \eta^{-1/2}$ also has measure${}\leq \eta\,\Vert f_j\Vert^2$
and outside  $N''_j$ we have, for all $t\in[t_{j-1},t_j]$,
$$\vert \du(\gamma(t_j),g)f_j(x)-\du(\gamma(t),g)f_j(x)\vert< \eta^{1/2}.$$

Let $N_{\eta}$ be the union
of all $N'_j$ and $N''_j$ ($1\leq j\leq k$). Since $(f_t)$ has orthogonal
increments,
$$\mu(N_{\eta})\leq2\eta\,\sum_{j=1}^k\Vert f_j \Vert_2^2
=2\eta\,\Vert f\Vert_2^2,$$
and, for each $x\notin N_{\eta}$,
we have for every $t\in L$, if
 $t\in[t_{j-1},t_j]\inter L$,
\begin{equation}\label{notin}
    \vert f_j(T_gx)-\du(\gamma(t),g)f_j(x)\vert< 2\eta^{1/2}.
\end{equation}
Let $(\eta_n)$ be a sequence of positive reals converging to 
$0$ and let $F$ be the set of $x\in X$ where the sample functions are 
c\`{a}dl\`{a}g and which belong to infinitely many sets 
$[0,1]\setminus N_{\eta_n}$. Then $F$ has full measure and, if $x\in 
F$, (\ref{notin}) holds for an infinite subsequence of $(\eta_n)$ with
the corresponding $j=j_n(t)$. As $f_{j_n(t)}$ converges pointwise to 
the jump $\Delta f_t$, it follows that (\ref{sauts}) holds for every $x\in F$.

Now, there is a set of full measure on which (\ref{sauts}) holds for all
$g$ in $G$. Then the modulus $\vert \Delta f_t(x)\vert$ of the jump is
invariant under the action of $G$, for all $t\in L$. The rest of the proof
is exactly as in \cite{Zs}: by ergodicity, the number of jumps of 
modulus${}\geq\delta>0$ in any given interval must be a.e.\ constant, 
whence non-zero jumps can only happen at fixed points of the interval,
and thus correspond to eigenfunctions in the closed invariant subspace
generated by $f$, which would contradict the hypothesis that $\si_f$ is 
continuous.
\end{proof}
\subsection{A reduction}
The following lemma expresses a criterion for $f$ to be Gaussian, 
independently of the construction of a spectral process.
\begin{lemma}\label{disj} Assume that $(T_g)$ is ergodic and $\si_f$ 
continuous. In order that $f$ be Gaussian, it is sufficient that, for 
every disjoint open sets $U$ and $V$ in $\Gamma$, the factors generated 
by $\pi_Uf$ and $\pi_Vf$ be independent.
\end{lemma}
\begin{proof} Assume that this condition holds. By Lemma \ref{FSlimits} and Lemma \ref{speccont}, it is enough to show that, for each totally disconnected
compact set $K$ in $\Gamma$, the corresponding process $(f_t)$
has independent increments. Now, for every $t\in(0,1)$, by the hypothesis and
the first assertion of Lemma \ref{speccont}, the factors generated by $f_t$
and $f_1-f_t$ are independent. A priori, we would have to show that $(f_{t_j}-f_{t_{j-1}})$
is independent for every finite sequence $0<\cdots<t_j<\cdots<1$ but, as given any
$t\in(0,1)$, all $f_{t'}$ with $t'<t$ belong to the closed invariant space
$Z(f_t)$, the conclusion follows by induction.
\end{proof}
Of course, by regularity of $\si_f$, the same condition with disjoint
compact sets instead of open sets is also sufficient, and in fact it
will hold for any pair of disjoint Borel sets. Notice that the condition 
for symmetric Borel sets is also necessary in an invariant Gaussian 
space, since two orthogonal real functions in a Gaussian space are 
independent. 
\section{Group property and Carleman's condition}\label{group}
We will now prove Theorem \ref{CI}.
We need the following result of Foia\c{s} \cite{Foias1}  (see also \cite{Foias2}), which states that {\em topological supports} of spectral measures of locally compact group actions in ergodic theory satisfy  a ``group property'':
\begin{theorem*}[Foia\c{s}]
Let $f_1$, $f_2\in L^2(\mu)$. If $f_1f_2\in L^2(\mu)$ then
$$\supp(\si_{f_1f_2})\subset \supp(\si_{f_1})\cdot\supp(\si_{f_2}).$$
\end{theorem*} 

\begin{lemma}\label{sectgroup}
 Let $f\in \CC(\mu)$ and let $U$ be an open set in $\Gamma$. The spectral type of $T$ restricted to the factor generated by $\pi_U f$ is concentrated on $\Gp(U\inter\supp(\si_f))$.
\end{lemma}

\begin{proof} The space of functions in $A(\Gamma)$ with compact support contained in $U$ is dense in the space of all continuous functions on $\Gamma$ vanishing outside $U$, so we can choose a sequence $(\vphi_n)$ of functions in $A(\Gamma)$ of norm $\Vert\vphi_n\Vert_{A(\Gamma)}\le 1$ with $\supp(\vphi_n)\subset U$ which span a dense subspace of $L^2(\sigma_f|_U)$. The spectral isomorphism yields that $(\vphi_n(T)f)$ span a dense subspace of $Z(\pi_U f)=\pi_U Z(f)$.

Let then $(f_n)_{n\ge 1}$ be the sequence of all functions $\Re \vphi_n(T)f$ and $\Im \vphi_n(T)f$, reordered. It is a sequence of real functions generating a dense subspace of $Z(\pi_U f)+Z(\overline{\pi_U f})$. For every $n\ge 1$, we have $\Vert f_n\Vert_p\le\Vert f\Vert_p$ for all $p\ge 2$, $f_n$ belongs to $\CC(\mu)$ and has a spectral measure $\si_{f_n}\ll\si_f|_U+\til\si_f|_{U^{-1}}$, so that  $\supp(\si_{f_n})\subset \Gp(U\inter\supp(\si_f))$.

Foia\c{s}' theorem applies to their finite products, which all belong to $L^2(\mu)$, so the spectral measure of any finite product of the $f_n$ is concentrated on $\Gp(U\inter\supp(\si_f))$; since the spectral measure of a sum is absolutely continuous with respect to the sum of the spectral measures of its terms, this remains true by linearity for all polynomials in the functions $f_n$.

The result will follow if we show that these polynomials are dense in $L^2\bigl(\B(Z(\pi_U f))\bigr)$. Moreover, as $\B(Z(\pi_U f))$ is the sub-$\si$-algebra generated by $\{f_n\}_{n\ge1}$, the union of the subspaces $L^2(\B\bigl(\{f_j\}_{1\le j\le n})\bigr)$ is dense in $L^2\bigl(\B(Z(\pi_Uf))\bigr)$.
Thus, it is enough to show that for every given $n\ge 1$ the polynomials in $f_1$,\ldots, $f_n$ are dense in  $L^2(\B\bigl(\{f_j\}_{1\le j\le n})\bigr)$.

This is a classic result for the case of a single function satisfying Carleman's condition, and we just need to extend it. As in this case, we shall use quasi-analytic classes, for which we refer to W.~Rudin's book \cite{Rud_R}, Chap.\ 19.

Fix $n\ge 1$ and let $\nu$ be the joint distribution of $f_1$,\ldots, $f_n$, so that for any positive measurable function $h$ on $\R^n$
$$\int_{\R^n} h(t_1,\ldots,t_n)\,d\nu(t_1,\ldots,t_n)=\int_{X} h\bigl(f_1(x),\ldots,f_n(x)\bigr) d\mu(x).$$
Under the map $h\to h\circ(f_1,\ldots,f_n)$, $L^2(\B\bigl(\{f_j\}_{1\le j\le n})\bigr)$ is isomorphic to $L^2(\nu)$, the $f_j$ corresponding to the coordinate functions, so we have to show that the polynomials are dense in $L^2(\nu)$.

Let $h$ be a function of $L^2(\nu)$ orthogonal to all polynomials, and consider the Fourier transform given on $\R^n$ by
$$\Phi(s_1,\ldots,s_n)=\int_{\R^n}\exp \Bigl(i\sum_{j=1}^n s_jt_j\Bigr)\,h(t_1,\ldots,t_n)\,d\nu(t_1,\ldots,t_n).$$
For every sequence $k_1,\ldots,k_n$ of natural integers,
\begin{equation*}
\int_{\R^n}\vert t_1^{k_1}\cdots t_n^{k_n}\vert^2\,d\nu(t_1,\ldots,t_n)=\int_{X}\vert f_1^{k_1}\cdots f_n^{k_n}\vert^2\,d\mu
\le \prod_{j=1}^{n} \Vert f_j^{2k_j}\Vert_n
\le \prod_{j=1}^{n} \Vert f\Vert_{2nk_j}^{2k_j},
\end{equation*}
and
\begin{equation*}
\int_{\R^n}\vert t_1^{k_1}\cdots t_n^{k_n}\vert\, h(t_1,\ldots,t_n)\,d\nu(t_1,\ldots,t_n)\le \prod_{j=1}^{n} \Vert f\Vert_{2nk_j}^{k_j}\cdot\Vert h\Vert_2.
\end{equation*}
Hence $\Phi$ is $C^{\infty}$. We denote by $D^{k_1,\ldots,k_n}\Phi$ the derivative $\partial^{k_1+\cdots+k_n}\Phi/\partial^{k_1}s_1\cdots\partial^{k_n}s_n$,
\begin{equation*}
D^{k_1,\ldots,k_n}\Phi(s_1,\ldots,s_n)=\int_{\R^n}\prod_{j=1}^n (it_j)^{k_j}\exp \bigl(i\sum_{j=1}^n s_jt_j\bigr)\, h(t_1,\ldots,t_n)\,d\nu(t_1,\ldots,t_n).
\end{equation*}
All these derivatives vanish at $(0,\ldots, 0)$ and we have have the bound
\begin{equation}\label{majderiv}
\vert D^{k_1,\ldots,k_n}\Phi(s_1,\ldots,s_n)\vert \le \prod_{j=1}^{n} \Vert f\Vert_{2nk_j}^{k_j}\cdot\Vert h\Vert_2\quad\mbox{for all }s_1,\ldots,s_n.
\end{equation}
Given $j$ $(1\le j\le n$), for every choice of $k_{\ell}$ ($\ell> j$) and $s_{\ell}$ ($\ell< j$), we consider the one-variable function
\[s\to D^{0,\ldots,0,k_{j+1},\ldots,k_n}\;\Phi(s_1,\ldots,s_{j-1},s,0,\ldots,0),\]
and the sequence of its derivatives. By (\ref{majderiv}), the $L^{\infty}$ norm of its $k$-th derivative is bounded by
$$
\Bigl(\prod_{\ell=j+1}^n \Vert f\Vert_{2nk_\ell}^{k_\ell}\Vert h\Vert_2\Bigr)\cdot M_k
$$
where $M_k=\Vert f\Vert_{2nk}^{k}$, which means, according to the definition in \cite{Rud_R}, that these functions belong to the class  $C\{M_k\}$.

Now, by standard application of the Hölder inequality, the sequence $(\Vert f\Vert_p)$ is non-decreasing and $(\Vert f\Vert_p^p)$ is logarithmically convex. Logarithmic convexity is inherited by the arithmetic subsequence $(\Vert f\Vert_{2nk}^{2nk})$ and then by $(\Vert f\Vert_{2nk}^{k})$, and the assumption that $f$ belongs to $\CC(\mu)$ together with the monotony of $(\Vert f\Vert_p)$ implies that
$$
\sum_{k=1}^{\infty}(1/M_k)^{1/k}=\sum_{k=1}^{\infty}1/\Vert f\Vert_{2nk}=+\infty.
$$
It follows from the Denjoy-Carleman Theorem that the class $C\{M_k\}$ is quasi-analytic, and thus that if one of these functions vanishes at $0$ together with all its derivatives then it is identically $0$.

We conclude by an easy induction. For $j=1$, given any $k_2$,\ldots, $k_n$, all the functions $s\to D^{k,k_2,\ldots,k_n}\;\Phi(s,0,\ldots,0)$ for $k\ge 0$ vanish at $0$, whence $D^{0,k_2,\ldots,k_n}\;\Phi(s_1,0,\ldots,0)=0$ for all $s_1$ and all $k_2$,\ldots, $k_n$.
Then, for a given $j\ge1$, if we have that 
\[D^{0,\ldots,0,k_{j+1},\ldots,k_n}\;\Phi(s_1,\ldots,s_{j},\,0,\ldots, 0)=0\quad\mbox{for all }s_1,\ldots, s_{j} \mbox{ and all }k_{j+1},\ldots, k_n,\]
we get similarly that all functions $s\to D^{0,\ldots,0,\,k_{j+2},\ldots,k_n}\;\Phi(s_1,\ldots,s_{j},s,\,0,\ldots, 0)$ are identically $0$ and the induction hypothesis remains true for $j+1$ instead of $j$.

Finally, $\Phi$ itself is identically $0$ and this implies $h=0$ $\nu$-a.e. So, the null function is the only function in $L^2(\nu)$ orthogonal to all polynomials, whence polynomials are dense in $L^2(\nu)$, and the proof is complete.
\end{proof}

\begin{proof}[Proof of Theorem \ref{CI}]
 Assume that $T$ is ergodic, let $f$ be a non-zero function in $\CC(\mu)$ whose spectral measure is continuous and concentrated on an independent compact set, and let $U$ and $V$ be two disjoint open sets in $\Gamma$.  By Lemma \ref{group}, the spectral types of $T$ on the factors generated by $\pi|_U f$ and $\pi|_V f$ are concentrated on $\Gp(U\inter\supp(\si_f))$ and $\Gp(V\inter\supp(\si_f))$ respectively. Since $\supp(\si_f)$ is independent these groups have no other common element than $1$. So, these factors are spectrally disjoint and a fortiori independent. The conclusion follows then from Lemma \ref{disj}.
\end{proof}

\section{The spectral projector on a Helson set}\label{sectHelson}
To prove Theorem \ref{Helsonproj} we need to approximate the indicator function of a compact Helson set by functions in $A(\Gamma)$ (here, we release the assumption that $\Gamma$ itself is compact). The main tool is Drury's lemma, which was used to solve the problem of the union of two Sidon sets \cite{Drury}, and that of the union of two Helson sets by N.~Varopoulos \cite{Varo1}, \cite{Varo2}. We quote it in the version given by C.~Herz \cite{Herz} with a more convenient estimate of the norm of the functions obtained in $A(\Gamma)$.

\begin{definition*} Let, for $0<\eps\le 1$,
$$\omega(\eps)=\sup_{n\ge 1}\omega_n(\eps),$$
where, denoting by $E_n$ the canonical basis of $\Z^n$,
$$\omega_n(\eps)=\inf_{\psi\in A(\Z^n)}\{\Vert\psi\Vert_{A(\Z^n)}:\;\psi|_{E_n}=1,\;\vert\psi\vert\le\eps \mbox{ on }\Z^n\setminus E_n\}.$$
\end{definition*}
\begin{theorem*}[Drury, Varopoulos, Herz] Let $K$ be a compact Helson--$\alpha$ set of $\Gamma$. For each closed set $F$ of $\Gamma$ disjoint from $K$, for every $\eps\in(0,1]$ and every $\beta<\alpha^{2}$, there exists a function $\vphi$ in $A(\Gamma)$ such that 
$$\vphi=1\mbox{ on }K,\ \vert\vphi\vert\le\eps\mbox{ on }F\mbox{ and }\Vert\vphi\Vert_{A(\Gamma)}\le \beta^{-1}\omega(\beta\eps)$$
\end{theorem*}
\begin{remark}\em This is essentially Theorem 2 of \cite{Herz}, except that it only states that there exists a function $\omega:(0,1]\to [1,+\infty)$ with that property. C.~Herz defines $\omega$ later on (after the statement of Proposition 1), as above, and he shows that it is suitable for the theorem. Besides the Helson constant is inverted there.
\end{remark}
However, the estimates of $\omega(\eps)$ given in \cite{Herz}
do not seem sufficient to show Theorem \ref{Helsonproj}. A little later, in a paper \cite{Mela} on a slightly different problem, J.-F. Méla gave indirectly a nearly optimal bound. 

\begin{theorem*}[M\'ela]For all $\eps$ in $(0,1/2]$
$$\omega(\eps)\le 2\vert\log\eps\vert+6.$$
\end{theorem*}
\begin{proof}As this result is not explicitly stated in \cite{Mela}, we explain briefly how to deduce it. Fix an arbitrary integer $n\ge 1$, and consider the construction in section 7  when the group $G$ is $\T^n$ and its dual is $\Z^n$. The measures noted $\nu_s$ in \cite{Mela} are then the finite Riesz products on $\T^n\times\T$ admitting as density with respect to the Lebesgue measure the positive trigonometric polynomials
$$Q_s(z_1,\ldots,z_n,z)=\prod_{j=1}^{n}\bigl(1+s(z_jz+\bar{z}_j\bar{z})\bigr),\quad s\in (0,1/2],$$
where $\Vert Q_s\Vert_{L^1(\T^{n+1})}=1$ for all $s$.

For $s\in (0,1/2]$, the measure $\mu_s$ on $\T^n$ is then defined by $\hat{\mu}_s(k_1,\ldots, k_n)=\hat{\nu}_s(k_1,\ldots, k_n,1)$. It admits as density the factor $P_s$ of $\bar z$ in the expansion of $Q_s$. On the canonical basis, its coefficients are all equal to $s$, all its other non-zero coefficients are odd powers $s^{2k+1}$ of $s$ with $k\ge 1$, and we still have $\Vert P_s\Vert_{L^1(\T^n)}\le 1$.

Now, the main idea is to construct, given $\eps>0$, a measure $\si$ on $(0,1/2]$ of norm as small as possible with $\int s\, d\si=1$ and $\vert\int s^{2k+1}d\si\vert\le \eps$ for all $k\ge 1$. This is achieved by Lemma 3 of  \cite{Mela}, where J.-F.~Méla shows that we can obtain $\Vert\si\Vert\le 2\vert\log\eps\vert+6$ (taking $a=\log2-1/2$ in the bound given in \cite{Mela}).

Then, integrating $P_s$ with respect to $\si$, we get a trigonometric polynomial $P$ on $\T^n$ whose Fourier transform $\psi$ on $\Z^n$ satisfies $\psi=1$ on $E_n$,  $\vert\psi\vert\le\eps$ elsewhere, and
$$\Vert\psi\Vert_{A(\Z^n)}=\Vert P\Vert_{L^1(\T^n)}\le \int\Vert P_s\Vert_{L^1(\T^n)}\,d\vert\si\vert\le 2\vert\log\eps\vert+6.$$
It follows that $\omega_n(\eps)\le 2\vert\log\eps\vert+6$ for all $n\ge 1$. 
\end{proof}
\begin{remark}\em We also mention that this result by J.-F.~Méla was used for a different problem in \cite{LRP}. The bound is nearly optimal since it can also be deduced from \cite{Mela} that $\omega(\eps)/\vert\log\eps\vert$ is bounded from below by a positive constant. 
\end{remark}
\begin{proof}[Proof of Theorem \ref{Helsonproj}]
Let a compact Helson set $K\subset\Gamma$ be given and denote by $\si$ be the spectral type of $T$.

A remarkable fact in the results by S.~Drury, N.~Varopoulos and C.~Herz is that the bound for $\Vert\vphi\Vert_{A(\Gamma)}$ does not depend on the set $F$ disjoint from $K$. Thanks to Méla's result, for $\eps\le 1/2$, this bound is less than $c\,\vert\log\eps\vert$, where $c$ is a constant depending only on the Helson constant of $K$.

If we fix $\eps\in(0,1/2]$ and appply this to a non-decreasing sequence $(F_n)$ of closed sets whose union is $\Gamma\setminus K$, we get a sequence $(\vphi_{n,\eps})$ of functions with $A(\Gamma)$ norms${}\le c\vert\log\eps\vert$, hence also bounded in $L^{\infty}(\si)$. By extracting a subsequence if needed, we may assume that it converges to some function $\vphi_{\eps}$ in the weak$^{\ast}$ topology of the duality ($L^1(\si),L^{\infty}(\si)$). Then the  operator $\vphi_{\eps}(T)$ on $L^2(\mu)$ is the weak limit of the sequence $(\vphi_{n,\eps}(T))$ and
$$\vphi_\eps=1\ (\si|_K)\mbox{--a.e.},\ \ \vert\vphi_\eps\vert\le\eps\ (\si|_{\Gamma\setminus K}) \mbox{--a.e.},$$

For $2\le p\le +\infty$, as each $\vphi_{n,\eps}(T)$ maps $L^p(\mu)$ into itself with operator norm${}\le c\,\vert\log\eps\vert$, we still have that $\vphi_{\eps}(T)$ maps $L^p(\mu)$ into itself, and
\begin{equation}\label{phisurLp}
\Vert\vphi_{\eps}(T)\Vert_{\LL(L^p(\mu))}\le c\,\vert\log\eps\vert.
\end{equation}

\vspace{1ex}
Let $(\eps_k)$ be a decreasing and summable sequence in $(0,1/2]$.
As $\vert \mathbf{1}_K-\vphi_{\eps_k}\vert\le \eps_k$ $\si$--a.e., we may write
$$\mathbf{1}_K=\vphi_{\eps_1}+\sum_{k=1}^{+\infty}(\vphi_{\eps_{k+1}}-\vphi_{\eps_k})\ \si\mbox{--a.e.},$$
where the series converges in the $L^\infty(\si)$ norm, and the corresponding series
\begin{equation}\label{serie}
\vphi_{\eps_1}(T)+\sum_{k=1}^{+\infty}(\vphi_{\eps_{k+1}}(T)-\vphi_{\eps_k}(T))
\end{equation}
converges towards $\pi_K$ in $\LL(L^2(\mu))$.

\vspace{1ex}
Moreover, for all $k\ge1$, as $\vert \vphi_{\eps_{k+1}}-\vphi_{\eps_k}\vert\le 2\eps_k$ $\si$--a.e.,
\begin{equation}\label{LL2}
\Vert \vphi_{\eps_{k+1}}(T)-\vphi_{\eps_k}(T)\Vert_{\LL(L^2(\mu))}=\Vert \vphi_{\eps_{k+1}}-\vphi_{\eps_k}\Vert_{L^\infty(\si)}\le 2\eps_k
\end{equation}
and, by (\ref{phisurLp}),
\begin{equation}\label{LLinfty}
\Vert\vphi_{\eps_{k+1}}(T)-\vphi_{\eps_k}(T)\Vert_{\LL(L^{\infty}(\mu))}\le 2c\,\vert\log\eps_{k+1}\vert,
\end{equation}

\vspace{1ex}
Now, for $p\in[2,+\infty)$, we have, by (\ref{LL2}), (\ref{LLinfty}) and the Riesz-Thorin interpolation theorem,
$$\Vert\vphi_{\eps_{k+1}}(T)-\vphi_{\eps_k}(T)\Vert_{\LL(L^{p}(\mu))}\le(2\eps_k)^{2/p}(2c\,\vert\log\eps_{k+1}\vert)^{1-2/p}.$$
With $\eps_k=\mathrm{e}^{-kp}$, we get
$$\Vert\vphi_{\eps_{k+1}}(T)-\vphi_{\eps_k}(T)\Vert_{\LL(L^{p}(\mu))}\le 2^{2/p}\,\mathrm{e}^{-2k}\cdot 2c(k+1)\,p\le 4c(k+1)\mathrm{e}^{-2k}\,p.$$
Il follows that the series (\ref{serie}) converges in $\LL(L^p(\mu))$. As $L^p(\mu)\subset L^2(\mu)$, this proves that $\pi_K$ maps $L^p(\mu)$ into itself, and that
$$\Vert\pi_K\Vert_{\LL(L^p(\mu))}\le c\,p+ 4c\Bigl(\sum_{k=1}^{+\infty}(k+1)\mathrm{e}^{-2k}\Bigr)\, p\le Cp,$$
where $C$ is a constant depending only on the Helson constant of $K$.

\vspace{1ex}
The last assertion of the theorem is an immediate consequence.
\end{proof}
\begin{proof}[Proof of Theorem \ref{HI}]
Assume again that $G$ is discrete and $\Gamma$ compact, and that $T$ is ergodic. Let $f$ be a non-zero function in $L^2(\mu)$ whose spectral measure
is continuous and concentrated on an independent Helson set $K$ of $\Gamma$.

By theorem \ref{Helsonproj}, $\pi_K$ maps $L^{\infty}(\mu)$ into $\CC(\mu)$ and thus there is a dense subspace of $\pi_K (L^2(\mu))$ consisting of functions in $\CC(\mu)$. In case when $K$ does not contain any eigenvalue of $T$, these functions have continuous spectral mesures and by Theorem \ref{CI} therefore have a Gaussian distribution. It then follows that every function in $\pi_K (L^2(\mu))$ is Gaussian.
Otherwise, as $\si_f$ is continous, we can choose a sequence $(K_n)$ of closed subsets of $K$ which do not contain any eigenvalue such that 
$\si_{f}(K\setminus K_n)\to 0$. Then, each $\pi_{K_n}f$ is Gaussian and it follows again that $f$ is Gaussian.
\end{proof}
\begin{proof}[Proof of Corollary \ref{process}]
We have to prove that $Z(f)$ is a Gaussian space and, in general, a subspace $H$ of $L^2(\mu$) is Gaussian if every non-zero function in $H$ has a Gaussian distribution. Under the assumptions of Theorem~\ref{HI} or of Theorem~\ref{CI}, every function $h$ in $Z(f)$ has a continuous spectral measure concentrated on $K=\supp(\si_f)$. If $K$ is a Helson set, it follows directly from Theorem~\ref{HI} that $h$ has a Gaussian distribution. If $f\in\CC(\mu)$, we can apply Theorem \ref{CI} to functions $\vphi(T)f$ with $\vphi\in A(\Gamma)$, which all belong to $\CC(\mu)$, and the result follows since they are are dense in $Z(f)$.
\end{proof}

\goodbreak
\section{Complements}\label{compl}

\subsection{Foia\c{s} and Stratila measures and sets}\label{FSmes} We say that a positive continuous measure $\si$ on
$\Gamma$ is a \emph{$FS$ measure} if, whenever the measure--preserving
action $(T_g)_{g\in G}$ on $(X,\B,\mu)$ is ergodic and $f$ is a complex 
function in $L^2(\mu)$ with $\si_f=\si$, then $f$ is Gaussian, and that
a Borel subset $K$ of $\Gamma$ is a \emph{$FS$ set} if every positive
continuous measure concentrated on $K$ is a $FS$ measure. For 
symmetric sets or measures, these definitions match the definitions 
of \cite{LP} and \cite{LPT}.

Recall that a compact subset $K$ of $\Gamma$ is a \emph{Kronecker} set if
every continuous function of modulus one on $K$ is a uniform limit
of characters. 
By Lemma \ref{FSlimits} (see also \cite{LP}), any positive measure $\si$ on $\Gamma$
concentrated on a non-decreasing union of $FS$ sets is a $FS$ measure.
It follows that the Foia\c s and
Stratila theorem still holds for a \emph{weak Kronecker} set in $\T$, that is
a closed subset $K$ such that every finite Borel measure $\si$ on $K$ is
concentrated on a non-decreasing union of Kronecker sets.
In the same way, the assumption on $\si_{f}$ in Theorem \ref{HI} can be
weakened: it is sufficient for $\si_{f}$ to be continuous and concentrated
on a non-decreasing union of independent Helson sets.

Now, a weak Kronecker set is a Helson--$1$ set and conversely a Helson--$1$
set is the translate of a weak Kronecker set (\cite{LiPo}, Chap.\ XIII).
Moreover a weak Kronecker set is independent (in the strong sense). So, 
Theorem \ref{HI} extends the Foia\c s and Stratila theorem, 
and the actual extension consists in the case of Helson--$\alpha$ sets with 
$0<\alpha<1$.

\vspace{1ex}
Let us also recall the following result of \cite{LP} and \cite{LPT}
for measures on $\T$, which shows that in this case $\si_{f}$ need not be concentrated 
on an independent set:
\begin{proposition*}
\emph{(1)} If $\si$ is a $FS$ measure and $0<\tau\ll\si$, then $\tau$ is a 
$FS$ measure.

\noindent\emph{(2)} Assume that $\si_1$ and $\si_2$ are mutually singular 
symmetric $FS$ measures. Then $\si_1+\si_2$ is a $FS$ measure if and only
if $\si_2$ is singular with respect to each translate of $\si_1$.
\end{proposition*}
\begin{remark}\em The only difficult point in this proposition is the
``if" part of (2) (\cite{LPT} Corollary 10): it is
easy to see that $\si_1+\si_2$ is $FS$ if and only if $T_{\si_1}$ and $T_{\si_2}$ are 
disjoint but the result then requires the characterization of
disjointness for $GAG$ automorphisms established in \cite{LPT}.
\end{remark}

\vspace{2ex}
\subsection{Mildly mixing example}\label{mildmixing}

The Helson hypothesis forbids mixing for the corresponding Gaussian
systems, since the upper bound $\sup_{g\in G}\vert \hat\si(g)\vert$ in the definition of a Helson set $K$ can be 
replaced by the upper limit at infinity (with a different constant, \cite{LiPo}, 
Chap.\ I, Prop.\ 5.2).
When $\alpha=1$, in particular when $K$ is a Kronecker set, it is easy to see that, given any positive Borel measure $\sigma$ on $K$, there are non-trivial 
sequences of characters converging to $1$ in $L^1(\sigma)$ (i.e.\ $K$ is a \emph{weak Dirichlet
set}), which, in our context, implies rigidity: there are
sequences $(T_{g_j})$ where $g_j\to\infty$ converging to the identity on the factor
generated by $f$. In \cite{LP} non rigid examples of $FS$ sets are
constructed by considering independent unions of Kronecker sets, but then
the Gaussian system is generated by rigid factors.

On the contrary, Theorem \ref{HI} allows mild mixing, that is absence of non-trivial rigid factors, and even a 
spectral form of partial mixing (usual partial mixing never occurs for Gaussian automorphisms which are not strongly
mixing).
Indeed, by a result of T. K\"orner \cite{Ko1} (see \cite{LiPo}, Chap.\ XIII, Theorem 3.14),
for each $\alpha$ in $(0,1)$, there exists a finite positive measure $\si$, concentrated on an
independent Helson--$\alpha$ set $K$ of $\T$, with the property 
$$\text{\em for every Borel set $B\subset \T$},\ 
\limsup\vert\hat{\si|_B}(n)\vert=\alpha\,\si(B).
$$
Then $\si$ is continuous, since $\alpha<1$, and $\til\si$ shares the same
property.
It follows that, for every positive
integer $k$ and  every positive measure $\tau\ll(\si+\til\si)^{\ast k}$,
$\limsup \vert\hat\tau(n)\vert\le \alpha^k\Vert\tau\Vert$ (by 
density, it is sufficient to check this inequality for $\tau$ of the 
form $\si_1|_{B_1}\ast\cdots\ast\si_k|_{B_k}$ where $\si_j=\si$ or $\si_j=\til\si$ and $B_j$ is a Borel subset of $\T$, for $1\le j \le k$).

Consider the Gaussian automorphism $T_{\si+\til\si}$, with spectral measure $\si_H=\si+\til\si$ on its Gaussian space $H$. 
The spectral type
of $T$ on the $n$-th chaos $H^{(n)}$ is the $n$-th convolution power
$\si_{H}^{\ast n}$, and the spectral type on the factor $\B(H)$ is the
convolution exponential $\exp(\si_{H})$, so we obtain:
\begin{corollary} For every $\alpha$ in $(0,1)$ there exists a 
$FS$ measure $\si$ on $\T$ such that the Gaussian automorphism $T=T_{\si+\til\si}$ 
satisfies the property: for every square integrable zero-mean function
$h$,
\[
\limsup\vert\scal(T^nh,h)\vert=
\limsup \vert\hat\si_h(n)\vert\leq \alpha\,\Vert h\Vert_2^2,
\]
and in particular $T$ is mildly mixing.
\end{corollary}

\subsection{Independence in measure}\label{mes_indep}
From now on, it will be more convenient to restrict ourselves to the action of a single automorphism, so $G=\Z$, and $\Gamma=\T$ (identified with $\mathbb{S}^1$).

A natural question is whether the assumption of support independence in Theorem \ref{HI} can be replaced by a notion of independence ``in measure''. Such a property appears in the spectral analysis of Gaussian automorphisms.
If  $\si$ is a continuous symmetric measure on $\T$, then $T_{\si}$ has simple spectrum iff the following condition holds: 
\begin{equation}\label{simplemult}
\begin{minipage}[t]{.9\textwidth}
\em For all $n\ge 1$, there exists a set of full $\si^{\otimes n}$--measure on which the product map $(z_1,\ldots, z_n)\to z_1\cdots z_n$ is one-to-one modulo permutations of coordinates.
\end{minipage}
\end{equation}
Indeed, this is equivalent to saying that for every $n\ge 1$ the cartesian power $T_{\si}^{\otimes n}$ restricted to the sub-$\si$-algebra of consisting of sets invariant under permutations of coordinates has simple spectrum, hence that $T_{\si}$ restricted to the $n$-th chaos has simple spectrum, and it also implies that the convolution powers of $\si$ are mutually singular (see e.g.\  \cite{KuPa}), i.e.\ that the chaos are spectrally disjoint.

In \cite{Age00}, O.~Ageev proved property (\ref{simplemult}) for the reduced spectral type (i.e.\ the spectral type restricted to zero-mean functions) of a class of rank one transformations, including the classical Chacon transformation, and it is clear that not all zero-mean functions in $L^2(\mu)$ can be Gaussian. So, Theorem \ref{HI} fails on the only assumptions that $T$ be ergodic and $\si_f$ be a continuous measure satisfying (\ref{simplemult}).

\vspace{1ex}
In the other direction, Theorem \ref{HI} implies that the spectral type of the Chacon transformation cannot be concentrated on an independent Helson set. But we don't know if it is concentrated on some independent set, and we leave open the question whether Theorem \ref{HI} is valid if we only assume that $\si_f$ is continuous and concentrated on an independent set. We do not know either if the result holds under the assumption that $\si_f$ is continuous, satisfies (\ref{simplemult}) and is concentrated on a Helson set.
\subsection{Poisson suspensions}\label{Poisson}
Poisson suspensions are particularly interesting in our context, because they appear together with Gaussians systems in the theory of processes with independent increments and their spectral properties are similar. We recall briefly results already discussed by E.~Roy in his article \cite{Roy}, to which we refer for a detailed exposition (see also \cite{DFLP}).

Let $T$ be a measure-preserving automorphism of a standard space $(X,\B,\mu)$ where $\mu$ is $\si$-finite and infinite. Suppose that there is no invariant set $E\in\B$ with $0<\mu(E)<+\infty$. This is equivalent to saying that its spectral type $\si$, which can be assumed to be symmetric, is continuous. Then its Poisson suspension $T_{\ast}$ is ergodic and spectrally isomorphic to the Gaussian automorphism $T_{\si}$.

More precisely, its $L^2$-space has a similar decomposition in an orthogonal sum $\oplus_{n\ge 0}H^{(n)}$ of chaos where, up to a normalizing constant, $H^{(n)}$ is isometric to the tensor product $L^2(\mu)^{\otimes n}$ restricted to functions invariant under coordinate permutations, the action of $T_{\ast}$ on $H^{(n)}$ being conjugate to $T^{\otimes n}$ and thus admitting the convolution power $\si^{\ast n}$ as spectral type. The spectral isomorphism with $T_{\si}$ sends each chaos $H^{(n)}$ of the suspension onto the corresponding chaos of the Gaussian system.

\vspace{1ex}
Poisson suspensions may have simple spectrum. Ageev's proof can be extended to some infinite measure preserving transformations obtained by \emph{cutting and stacking}, and such systems never have invariant sets of finite positive measure. In \cite{DanRyz}, A.~Danilenko and V.~Ryzhikov construct examples where moreover the Poisson suspension is mixing.

We thus obtain Poisson suspensions $T_{\ast}$ with simple spectrum whose spectral type restricted to the chaos $H^{(1)}$ is the spectral type $\si$ of $T$.
However, E.~Roy proves a disjointness property between Gaussian systems and Poisson suspensions (\cite{Roy}, Prop.\ 4.11). In particular no function from the chaos $H^{(1)}$ can have a Gaussian distribution, so $\si$ cannot be a $FS$ measure.

\vspace{1ex}
More generally the spectral type of a $\si$-finite measure preserving automorphism $T$  must be singular to every positive $FS$ measure whenever there is no invariant set $E\in\B$ with $0<\mu(E)<+\infty$ (\cite{Roy}, Theorem 4.13). We have the following consequence:

\begin{corollary}Let $T$ be a measure-preserving automorphism of a space $(X,\B,\mu)$ of $\si$-infinite measure with no invariant set of non-zero  finite measure, and let $\si$ be its spectral type. Then $\si(K)=0$ for every independent Helson set $K\subset\T$.
\end{corollary}

\begin{remark}\em These latter results can easily be extended to actions of locally compact second countable groups.

\end{remark}

\vspace{3ex}
\begin{ack*} I am grateful to Herv\'e and Martine Queff\'elec for suggesting that the class of Helson sets was the right class to look for a result like Theorem \ref{Helsonproj}.
I would also like to thank Mariusz Lema\'nczyk,  Emmanuel Roy, Jean--Paul 
Thouvenot and Benjamin Weiss for their interest in this work and helpful comments. 
\end{ack*}

\bibliographystyle{plain}
\bibliography{biblio_FS}

\end{document}